\pdfoutput=1 

\documentclass[11pt]{amsart}

\usepackage{amssymb,amsthm,enumitem,colonequals,tikz-cd,microtype}
\usepackage[normalem]{ulem} 

\usepackage[osf]{Baskervaldx}
\usepackage[baskervaldx]{newtxmath}
\usepackage[cal=boondoxo]{mathalfa} 

\usepackage[top=3.75cm, bottom=3cm, left=3.5cm, right=3.5cm]{geometry}

\usepackage{xcolor}
\colorlet{darkblue}{blue!55!black}
\colorlet{darkcyan}{cyan!50!black}
\colorlet{darkgreen}{green!60!black}

\PassOptionsToPackage{hyphens}{url}
\usepackage{hyperref}
\hypersetup{
    colorlinks=true,
    linkcolor=darkblue,
    urlcolor=darkcyan,
    citecolor=darkgreen,
}

\def\eqref#1{\textcolor{darkblue}{(\ref{#1})}}

\usepackage[nameinlink]{cleveref} 
\Crefformat{section}{#2\S#1#3}
\Crefmultiformat{section}{#2\S\S#1#3}{ and~#2#1#3}{, #2#1#3}{, and~#2#1#3}

\usepackage[pagewise]{lineno}
\let\oldequation\equation
\let\oldendequation\endequation
\renewenvironment{equation}{\linenomathNonumbers\oldequation}{\oldendequation\endlinenomath}
\expandafter\let\expandafter\oldequationstar\csname equation*\endcsname
\expandafter\let\expandafter\oldendequationstar\csname endequation*\endcsname
\renewenvironment{equation*}{\linenomathNonumbers\oldequationstar}{\oldendequationstar\endlinenomath}
\let\oldalign\align
\let\oldendalign\endalign
\renewenvironment{align}{\linenomathNonumbers\oldalign}{\oldendalign\endlinenomath}
\expandafter\let\expandafter\oldalignstar\csname align*\endcsname
\expandafter\let\expandafter\oldendalignstar\csname endalign*\endcsname
\renewenvironment{align*}{\linenomathNonumbers\oldalignstar}{\oldendalignstar\endlinenomath}

\theoremstyle{plain}
\newtheorem{theorem}{Theorem}[section]
\newtheorem{lemma}[theorem]{Lemma}
\newtheorem{corollary}[theorem]{Corollary}
\newtheorem{proposition}[theorem]{Proposition}

\theoremstyle{definition}

\newtheorem{remark}[theorem]{Remark}

\newtheorem*{ack}{Acknowledgments}
\newtheorem*{notation}{Notation}

\AddToHook{env/conjecture/begin}{\crefalias{theorem}{conjecture}}
\AddToHook{env/lemma/begin}{\crefalias{theorem}{lemma}}
\AddToHook{env/corollary/begin}{\crefalias{theorem}{corollary}}
\AddToHook{env/proposition/begin}{\crefalias{theorem}{proposition}}
\AddToHook{env/definition/begin}{\crefalias{theorem}{definition}}
\AddToHook{env/remark/begin}{\crefalias{theorem}{remark}}
\AddToHook{env/setup/begin}{\crefalias{theorem}{setup}}
\AddToHook{env/example/begin}{\crefalias{theorem}{example}}
\AddToHook{env/conjecture/begin}{\crefalias{theorem}{Conjecture}}

\numberwithin{equation}{section}
\numberwithin{theorem}{section}

\title[Measuring rationality of Schwede--Takagi pairs]{Measuring rationality of Schwede--Takagi pairs}

\author[P.~Lank]{Pat Lank}
\address{P.~Lank,
Dipartimento di Matematica “F. Enriques”, Universit\`{a} degli Studi di Milano, Via Cesare
Saldini 50, 20133 Milano, Italy}
\email{plankmathematics@gmail.com}

\author[P. ~McDonald]{Peter M. McDonald}
\address{P. ~McDonald, 
Department of Mathematics,
Simon Fraser University,
Burnaby, BC, V5A 1S6,
Canada
}
\email{pma94@sfu.ca}

\author[S.~Venkatesh]{Sridhar Venkatesh}
\address{S.~Venkatesh,
Department of Mathematics,
University of California, Los Angeles, 
Los Angeles, CA, 90095
U.S.A.}
\email{srivenk@math.ucla.edu}

\date{\today}



\begin{document}
    
\begin{abstract}
    We begin by giving a derived characterization of rational singularities for pairs in the sense of Schwede--Takagi. This characterization extends a characterization of rational singularities due to Lank--Venkatesh to pairs on normal varieties over fields of characteristic zero. As an application, we introduce a categorical invariant that measures the failure of rationality for pairs on affine varieties that are locally complete intersections.
\end{abstract}

\maketitle

\section{Introduction}
\label{sec:intro}

Throughout, by a \textit{variety} we mean a separated integral scheme $Y$ of finite type over a field of characteristic zero. The study of pairs $(Y,\mathcal{I}^c)$, where $Y$ is a variety, $\mathcal{I}$ is a coherent ideal sheaf on $Y$, and $c\in \mathbf{R}_{\geq 0}$, naturally arises in the minimal model program. In \cite{Schwede/Takagi:2008}, the notion of rational singularities was extended to pairs. 

Such extensions use a notion of \textit{log resolution} of a pair $(Y,\mathcal{I}^c)$. This is a proper birational morphism $f\colon \widetilde{Y}\to Y$ from a regular scheme such that $\mathcal{I} \cdot \mathcal{O}_{\widetilde{Y}} = \mathcal{O}_{\widetilde{Y}} ( -G)$ is a line bundle and $\operatorname{exc}(f)\cup \operatorname{supp}(G)$ is a simple normal crossings divisor. See \cite{Schwede/Takagi:2008, Kollar:1997,Kollar/Mori:2008} for details. 

Now, a pair $(Y,\mathcal{I}^c)$ has \textit{rational singularities} if the natural map $\mathcal{O}_Y \to \mathbf{R}f_\ast \mathcal{O}_{\widetilde{Y}} (\lfloor c \cdot G \rfloor)$ is an isomorphism for $f\colon \widetilde{Y} \to Y$ a log resolution of $(Y,\mathcal{I}^c)$ with $\mathcal{I}\cdot \mathcal{O}_{\widetilde{Y}} =  \mathcal{O}_{\widetilde{Y}} (-G)$. A key result, \cite[Theorem 3.11]{Schwede/Takagi:2008}, says that having rational singularities is equivalent to the natural map $\mathcal{O}_Y \to \mathbf{R}f_\ast \mathcal{O}_{\widetilde{Y}} (\lfloor c \cdot G \rfloor)$ splitting. In our work, we give a different characterization.

One approach to studying objects of $D^b_{\operatorname{coh}}$ is to use a concept of generation in triangulated categories and the associated numerics \cite{Bondal/VandenBergh:2003,Avramov/Buchweitz/Iyengar/Miller:2010}. Roughly speaking, given objects $E, G \in \mathcal{T}$ in a triangulated category, the \textit{level} of $E$ with respect to $G$, denoted $\operatorname{level}^G (E)$, is defined as follows. It is the smallest integer $n \geq 0$ such that $E$ can be obtained from $G$ using a finite number of shifts, direct summands, finite direct sums, and at most $n$ cones; it is set to $\infty$ otherwise.

These techniques have been used to give derived characterizations of singularities. For example, an affine scheme $X$ is locally a complete intersection if, and only if, for every $E \in D^b_{\operatorname{coh}}(X)$ there exists a nonzero $P \in \operatorname{Perf}(X)$ such that $\operatorname{level}^E (P) < \infty$ \cite{Pollitz:2019}. Other characterizations include rational singularities \cite{Lank/Venkatesh:2026} and Frobenius splitting \cite{BILMP:2023}; in fact, $\operatorname{level}$ was shown to measure the failure of such singularities in loc.\ cit. 

This brings attention to our first result.

\begin{theorem}
    \label{thm:schwede_takagi_pairs}
    Consider a pair $(Y,\mathcal{I}^c)$ with $Y$ normal. Let $f\colon \widetilde{Y}\to Y$ be a log resolution of $(Y,\mathcal{I}^c)$. Then $(Y,\mathcal{I}^c)$ has rational singularities if, and only if, $\mathcal{O}_Y \in \langle \mathbf{R} f_\ast \mathcal{O}_{\widetilde{Y}}(\lfloor c \cdot G \rfloor) \rangle_1$.
\end{theorem}

The singularities studied in \cite{Lank/Venkatesh:2026,BILMP:2023,DeDeyn/Lank/Lank/ManaliRahul/Venkatesh:2026} involve natural morphisms between structure sheaves, which are the units for derived push/pull adjunctions. However, the natural morphism $\mathcal{O}_Y \to \mathbf{R}f_\ast \mathcal{O}_{\widetilde{Y}} (\lfloor c \cdot G \rfloor)$ for rational pairs need not arise as the unit morphism for a pair of adjoint
functors.  So, our proof of \Cref{thm:schwede_takagi_pairs} requires a different strategy than what appears in \cite{Lank/Venkatesh:2026,BILMP:2023,DeDeyn/Lank/Lank/ManaliRahul/Venkatesh:2026}. We work around this by passing to the completions of the stalks, where we use local algebra to show that a characterization is possible. See proof of \Cref{thm:schwede_takagi_pairs} for details. Furthermore, \Cref{thm:schwede_takagi_pairs} extends \cite[Theorem 3.13]{Lank/Venkatesh:2026} to pairs on normal varieties.

Now, the question remains: Can we use $\operatorname{level}$ to measure the failure of rationality for pairs? In particular, we focus on $\operatorname{level}^{\mathbf{R} f_\ast \mathcal{O}_{\widetilde{Y}}(\lfloor c \cdot G \rfloor) } (\mathcal{O}_Y)$. One nuance in our setting is that this level may be infinite. Compare this to \cite{Lank/Venkatesh:2026,BILMP:2023,DeDeyn/Lank/Lank/ManaliRahul/Venkatesh:2026}, where the levels investigated are guaranteed to be finite due to the presence of so-called `nil-faithful algebras' in the sense of \cite[Definition 3.16]{Balmer:2016} (see also \cite[Definition 3.18]{Mathew:2016} for the $\infty$-categorical analog which is called `descendable algebras'). In particular, $\mathbf{R} f_\ast \mathcal{O}_{\widetilde{Y}}(\lfloor c \cdot G \rfloor)$ is not necessarily a descendable algebra, and so it remains an interesting question of when $\operatorname{level}^{\mathbf{R} f_\ast \mathcal{O}_{\widetilde{Y}}(\lfloor c \cdot G \rfloor) } (\mathcal{O}_Y)$ is even finite.

In what follows, we find a setting where finiteness holds. Let $f\colon \widetilde{Y}\to Y$ be a log resolution of $(Y,\mathcal{I}^c)$. Recall that the coherent sheaf is $f_\ast (\mathcal{O}_{\widetilde{Y}}(- \lfloor c \cdot G \rfloor) \otimes \omega_{\widetilde{Y}})$ is called the \textit{multiplier submodule} of $(Y,\mathcal{I}^c)$. It is denoted $\mathcal{J}(\omega_Y,\mathcal{I}^c)$. In the case $Y$ is Cohen--Macaulay, we show the value $\operatorname{level}^{\mathbf{R} f_\ast \mathcal{O}_{\widetilde{Y}}(\lfloor c \cdot G \rfloor)} (\mathcal{O}_Y)$ coincides with $\operatorname{level}^{\mathcal{J}(\omega_Y,\mathcal{I}^c)} (\omega_Y)$, see \Cref{lem:rationality_CM_via_multiplier_submodule}. 
This brings attention to the next result.

\begin{proposition}
    \label{prop:pairs}
    Consider a pair $(Y,\mathcal{I}^c)$ where $Y$ is normal, quasi-affine, and locally a complete intersection. Then $\operatorname{level}^{\mathcal{J}(\omega_Y, \mathcal{I}^c)} (\omega_Y)$ is finite. In particular, $(Y,\mathcal{I}^c)$ has rational singularities if, and only if, $\operatorname{level}^{\mathcal{J}(\omega_Y, \mathcal{I}^c)} (\omega_Y)=1$.
\end{proposition}

Let us focus on the case of normal affine varieties which are local complete intersections. Our motivation, as mentioned earlier, comes from the work of Pollitz \cite{Pollitz:2019}. We leverage the homotopical characterization for affine varieties that are local complete intersections to show that $\operatorname{level}^{\mathcal{J}(\omega_Y, \mathcal{I}^c)} (\omega_Y)$ is always finite. Hence, in such cases, we show that $\operatorname{level}^{\mathcal{J}(\omega_Y, \mathcal{I}^c)} (\omega_Y)$ measures the failure of the pair to be rational. Moreover, since $Y$ is affine, this value becomes computable using Macaulay2 packages such as \texttt{ThickSubcategories} \cite{Grifo/Letz/Pollitz:2025}.

It is important to clarify what these values indicate. If the pair has rational singularities, then so must $Y$ (see \cite[Corollary 3.13]{Schwede/Takagi:2008}). However, it may happen that $Y$ has rational singularities whereas the pair does not (e.g.\ \cite[Example 4.6]{Schwede/Takagi:2008}). Thus, these values measure the failure of the pair to be rational, quantifying a new relationship between the geometry of the ideal and that of the ambient variety by using the derived category. We believe these numerics to be valuable steps towards studying singularities of pairs via derived categorical methods.

In \Cref{app:kollar_kovacs_pairs}, we discuss a variation of our results in the setting of pairs \`{a} la Koll\'{a}r--Kov\'{a}cs \cite{Kollar:2013}. Specifically, we prove a version of \Cref{thm:schwede_takagi_pairs} for such pairs, see \Cref{thm:kollar_kovac_pairs}. Also, for lack of reference, we record a splitting criteria for rationality of such pairs that has been known to experts, see \Cref{prop:kovacs_splitting_rational_pairs}. 

\begin{notation}
    Let $X$ be a Noetherian scheme. Denote by $D(X):=D(\operatorname{Mod}(X))$ for the derived category of $\mathcal{O}_X$-modules; $D_{\operatorname{qc}}(X)$ for the strictly full subcategory of $D(X)$ consisting of complexes with quasi-coherent cohomology; $D_{\operatorname{coh}}^b(X)$ for the strictly full subcategory of $D(X)$ consisting of complexes having bounded and coherent cohomology; and $\operatorname{Perf}(X)$ for the strictly full subcategory of $D(X)$ consisting of perfect complexes. If $X = \operatorname{Spec}(A)$ is affine, then we at times write $D_{\operatorname{qc}}(A)$ for $D_{\operatorname{qc}}(X)$; similar conventions occur for the other categories. 
\end{notation}

\begin{ack}
   Lank was supported by the National
   Science Foundation under Grant No.\ DMS-2302263 and the ERC Advanced Grant 101095900-TriCatApp. Venkatesh was supported by the National
   Science Foundation under Grant No.\ DMS-2301463 and by the Simons Collaboration grant \textit{Moduli of Varieties}. The authors thank Takumi Murayama, S\'{a}ndor Kov\'{a}cs, and Karl Schwede for useful discussions on earlier versions of our work. Additionally, Lank thanks the University of Michigan for their hospitality during a visit for which part of this work was developed. Lastly, we greatly appreciate comments and suggestions from the referee.
\end{ack}

\section{Preliminaries}
\label{sec:prelim}

\subsection{Generation}
\label{sec:prelim_generation}

We discuss a notion of generation for triangulated categories and refer the reader to \cite{Bondal/VandenBergh:2003, Avramov/Buchweitz/Iyengar/Miller:2010}. Let $\mathcal{T}$ be a triangulated category with shift functor $[1]\colon \mathcal{T} \to \mathcal{T}$. Suppose $\mathcal{S}$ is a subcategory of $\mathcal{T}$. We say $\mathcal{S}$ is \textbf{thick} if it is a triangulated subcategory of $\mathcal{T}$ which is closed under direct summands. The smallest thick subcategory containing $\mathcal{S}$ in $\mathcal{T}$ is denoted by $\langle \mathcal{S} \rangle$; if $\mathcal{S}$ consists of a single object $G$, then $\langle \mathcal{S}\rangle$ will be written as $\langle G\rangle$. Denote by $\operatorname{add}(\mathcal{S})$ for the smallest strictly full (i.e.\ closed under isomorphisms) subcategory of $\mathcal{T}$ containing $\mathcal{S}$ which is closed under shifts, finite coproducts, and direct summands. Inductively, let
\begin{itemize}
    \item $\langle \mathcal{S} \rangle_0$ consists of all objects in $\mathcal{T}$ isomorphic to the zero objects
    \item $\langle \mathcal{S} \rangle_1 := \operatorname{add}(\mathcal{S})$.
    \item $\langle \mathcal{S} \rangle_n := \operatorname{add} \{ \operatorname{cone}(\phi) \colon \phi \in \operatorname{Hom}_{\mathcal{T}} (\langle \mathcal{S} \rangle_{n-1}, \langle \mathcal{S} \rangle_1) \}$.
\end{itemize}
It is straightforward to check that $\langle \mathcal{S} \rangle = \cup^\infty_{n=0} \langle \mathcal{S} \rangle_n$. If $E\in \mathcal{T}$, then the \textbf{level} of $E$ with respect $\mathcal{S}$, denoted $\operatorname{level}^\mathcal{S} (E)$, is defined as the smallest $n\geq 0$ such that $E\in \langle \mathcal{S} \rangle_n$; it is set to $\infty$ if there is no such $n$. We denote $\langle \mathcal{S}\rangle_n$ by $\langle G\rangle_n$ if $\mathcal{S}$ consists of a single object $G$; similarly for $\operatorname{level}$.

\subsection{Dualizing complexes}
\label{sec:prelim_dualizing_complexes}

We give a recap on dualizing complexes for schemes, and refer to \cite[\href{https://stacks.math.columbia.edu/tag/0DWE}{Tag 0DWE}]{StacksProject} or \cite{Hartshorne:1966}. Let $X$ be a Noetherian scheme. Suppose $R$ is a Noetherian ring. An object $\omega^\bullet_R$ in $D(R)$ is called a \textbf{dualizing complex} if the following conditions are satisfied: $\omega^\bullet_R$ has finite injective dimension, $\mathcal{H}^i(\omega^\bullet_R)$ is a finite $R$-module for all $i$, and $R \xrightarrow{ntrl.} \mathbf{R}\operatorname{Hom}_R (\omega^\bullet_R, \omega^\bullet_R)$ is an isomorphism in $D(R)$. This extends naturally to schemes as follows: an object $K$ of $D_{\operatorname{qc}}(X)$ is called a \textbf{dualizing complex} if it is affine locally a dualizing complex.

Let $K$ be a dualizing complex for $X$. Then $K$ belongs to $D_{\operatorname{coh}}^b (X)$ and the functor $\operatorname{\mathbf{R}\mathcal{H}\! \mathit{om}}_{\mathcal{O}_X}(-,K)\colon D_{\operatorname{coh}}^b (X) \to D_{\operatorname{coh}}^b (X)$ is an antiequivalence of triangulated categories. There are many useful relations among dualizing complexes,  derived sheaf hom functors, and derived pushforward functors. This falls under the term of `duality'. See e.g.\ \cite[\href{https://stacks.math.columbia.edu/tag/0DWE}{Tag 0DWE}]{StacksProject}. As a convention in our work, if $f\colon X \to \operatorname{Spec}(k)$ is a separated morphism of finite type to a field, we set $\omega_X^\bullet := f^! \mathcal{O}_{\operatorname{Spec}}(k)$ with $f^!$ the `upper shriek' functor. See \cite[\href{https://stacks.math.columbia.edu/tag/0AU3}{Tag 0AU3}]{StacksProject} for details.

\subsection{Schwede--Takagi rationality}
\label{sec:prelim_schwede_takagi}

We recall the notion of `rational pairs' introduced in \cite{Schwede/Takagi:2008}. A \textbf{pair} is a tuple $(Y,\mathcal{I}^c)$ where $Y$ is a variety, $\mathcal{I}$ is a coherent ideal sheaf on $Y$, and $c\in \mathbf{R}_{\geq 0}$. 
A \textbf{log resolution} for $(Y,\mathcal{I}^c)$ is a proper birational morphism $f\colon \widetilde{Y}\to Y$ from a smooth variety such that $\operatorname{exc}(f)$ is a divisor, $\mathcal{I} \cdot \mathcal{O}_{\widetilde{Y}} = \mathcal{O}_{\widetilde{Y}} ( -G)$ is an invertible sheaf with $G$ an effective divisor, and $\operatorname{exc}(f) + G$ is a simple normal crossings divisor.
Here $\operatorname{exc}(f)$ denotes the exceptional locus of $f$, which is defined as the closed subset $\widetilde{Y}\setminus U$ for $U$ the largest open subscheme over which $f$ is an isomorphism. 
We say $(Y,\mathcal{I}^c)$ has \textbf{rational singularities} if $\mathcal{O}_Y \xrightarrow{ntrl.} \mathbf{R}f_\ast \mathcal{O}_{\widetilde{Y}} (\lfloor c \cdot G \rfloor)$ is an isomorphism for $f\colon \widetilde{Y} \to Y$ a log resolution of $(X,\mathcal{I}^c)$ with $\mathcal{I}\cdot \mathcal{O}_{\widetilde{Y}} =  \mathcal{O}_{\widetilde{Y}} (-G)$. From \cite[Proposition 3.4]{Schwede/Takagi:2008}, the notion of rationality in their sense is independent of the log resolution for $(Y,\mathcal{I}^c)$. As a general fact for any pair, \cite[Lemma 3.5]{Schwede/Takagi:2008} states $\mathbf{R}^j f_\ast (\mathcal{O}_{\widetilde{Y}}(- \lfloor c \cdot G \rfloor) \otimes^{\mathbf{L}} \omega_{\widetilde{Y}}) = 0$ for $j>0$. Its $0$-th cohomology sheaf is $f_\ast (\mathcal{O}_{\widetilde{Y}}(- \lfloor c \cdot G \rfloor) \otimes \omega_{\widetilde{Y}})$, which is called the \textbf{multiplier submodule} of $(Y,\mathcal{I}^c)$ and is denoted $\mathcal{J}(\omega_Y,\mathcal{I}^c)$.

\subsection{Krull--Schmidt categories}
\label{sec:prelim_krull_schmidt}

We discuss ideas related to Krull--Schmidt categories. See e.g.\ \cite{Atiyah:1956, Walker/Warfield:1976}. Consider an additive category $\mathcal{C}$. We say $\mathcal{C}$ is \textbf{Krull-Schmidt} when every object in $\mathcal{C}$ decomposes into a finite coproduct of objects having local endomorphism rings. Here, an object of $\mathcal{C}$ is said to be \textbf{indecomposable} when it is not isomorphic to a direct sum of two nonzero objects. In fact, if $\mathcal{C}$ is a Krull-Schmidt category, then every object decomposes (uniquely up to permutations) into a finite coproduct of indecomposables, see \cite[Theorem 4.2]{Krause:2015}.

Let $R$ be a complete Noetherian local ring and an $R$-linear abelian category $\mathcal{A}$. We say $\mathcal{A}$ is \textbf{$\operatorname{Ext}$-finite} when the $R$-module $\operatorname{Ext}^n (A,B)$ is finitely generated for all objects $A,B$ in $\mathcal{A}$ and $n\geq 0$. It is useful to note that $\mathcal{A}$ is $\operatorname{Ext}$-finite over $R$ if, and only if, $D^b (\mathcal{A})$ is $\operatorname{Hom}$-finite over $R$. In fact, if $\mathcal{A}$ is $\operatorname{Ext}$-finite (over $R$), then $D^b (\mathcal{A})$ is a Krull-Schmidt category, see \cite[Corollary B]{Le/Chen:2007}.

The following was originally observed for the category of coherent sheaves on a proper variety over an algebraically closed field \cite[Theorem 2]{Atiyah:1956}, but generalized in \cite[Lemma 2.7]{Lank/Venkatesh:2026}.

\begin{lemma}
    \label{lem:krull_schmidt}
    If $X$ is a proper scheme over a Noetherian complete local ring $R$, then both $\operatorname{coh}X$ and $D^b_{\operatorname{coh}}(X)$ are Krull-Schmidt categories. Additionally, if $X$ is integral, then $\mathcal{O}_X$ is an indecomposable object in both categories.
\end{lemma}

\begin{remark}
    Let $\mathcal{T}$ be a Krull-Schmidt triangulated category and $E,G\in \mathcal{T}$. If $E\in \langle G \rangle_1$ and $E$ is indecomposable, then $E$ is a direct summand of $G$ up to a possible shift. Indeed, use that $G$ is a finite coproduct of indecomposables.
\end{remark}

\section{Results}
\label{sec:results}

Throughout, let $(Y,\mathcal{I}^c)$ be a pair where $Y$ is normal. We start by proving the main result.

\begin{proof}
    [Proof of \Cref{thm:schwede_takagi_pairs}]
    If $(Y,\mathcal{I}^c)$ has rational singularities, then $\mathcal{O}_Y \in \langle \mathbf{R} f_\ast \mathcal{O}_{\widetilde{Y}}(\lfloor c \cdot G \rfloor) \rangle_1$ because the category is closed under isomorphisms. So, we prove the converse direction. Before proving this, we make an observation. If $\mathcal{O}_Y \in \langle \mathbf{R} f_\ast \mathcal{O}_{\widetilde{Y}}(\lfloor c \cdot G \rfloor) \rangle_1$, then \cite[Theorem 3.13]{Lank/Venkatesh:2026} tells us $Y$ has rational singularities. Consequently, $Y$ must be Cohen--Macaulay, and so, $\mathcal{H}^j (\omega^\bullet_Y)\cong 0$ for $j\not=-\dim Y$.

    Now, we return to the proof. First, note that $\mathcal{O}_Y \xrightarrow{ntrl.} f_\ast \mathcal{O}_{\widetilde{Y}}(\lfloor c \cdot G \rfloor)$ is an isomorphism via \cite[\href{https://stacks.math.columbia.edu/tag/0AVS}{Tag 0AVS}]{StacksProject} as it is a nonzero morphism between torsion free sheaves of rank one which is an isomorphism in codimension one with $\mathcal{O}_Y$ being $(S_2)$. Thus, it suffices to show that $\mathbf{R}^j f_\ast \mathcal{O}_{\widetilde{Y}}(\lfloor c \cdot G \rfloor) = 0$ for $j>0$. 
    We claim that we just need to check that, for any $p\in Y$, $t^\ast \mathbf{R} f_\ast \mathcal{O}_{\widetilde{Y}}(\lfloor c \cdot G \rfloor)$ is concentrated in degree zero where $t\colon \operatorname{Spec}(\widehat{\mathcal{O}}_{Y,p}) \to Y$ is the natural morphism and $\widehat{\mathcal{O}}_{Y,p}$ is the completion of $\mathcal{O}_{Y,p}$ at its maximal ideal. To see why, note that this implies the stalk of $\mathbf{R} f_\ast \mathcal{O}_{\widetilde{Y}}(\lfloor c \cdot G \rfloor)$ at $p$ is concentrated in degree zero as completion is faithfully flat. Therefore, if $t^\ast \mathbf{R} f_\ast \mathcal{O}_{\widetilde{Y}}(\lfloor c \cdot G \rfloor)$ is concentrated in degree zero for all $p\in Y$, then $\mathbf{R}^j f_\ast \mathcal{O}_{\widetilde{Y}}(\lfloor c \cdot G \rfloor) = 0$ for $j>0$.

    For the remainder of the proof, fix $p\in Y$ and the morphism $t\colon \operatorname{Spec}(\widehat{\mathcal{O}}_{Y,p}) \to Y$. Recall from \cite[Lemma 3.5]{Schwede/Takagi:2008} that if $j\not=0$,
    \begin{displaymath}
        \mathbf{R}^j f_\ast (\mathcal{O}_{\widetilde{Y}}(- \lfloor c \cdot G \rfloor) \otimes^{\mathbf{L}} \omega_{\widetilde{Y}})\cong 0.
    \end{displaymath}
    In fact, the $0$-th cohomology is $f_\ast \mathcal{O}_{\widetilde{Y}}(- \lfloor c \cdot G \rfloor \otimes^{\mathbf{L}} \omega_{\widetilde{Y}})$, which is torsion free of rank one via \cite[Proposition 7.4.5]{EGAI:1971} and the fact $\mathcal{O}_{\widetilde{Y}}(- \lfloor c \cdot G \rfloor)$ is a line bundle. Using the natural isomorphism 
    \begin{displaymath}
        \mathbf{R} f_\ast (\mathcal{O}_{\widetilde{Y}}(- \lfloor c \cdot G \rfloor) \otimes^{\mathbf{L}} \omega_{\widetilde{Y}}) \to \operatorname{\mathbf{R}\mathcal{H}\! \mathit{om}} (\mathbf{R}f_\ast \mathcal{O}_{\widetilde{Y}} (\lfloor c \cdot G \rfloor) , \omega^\bullet_Y)[-\dim Y],
    \end{displaymath}
    we see that the only nontrivial cohomology of $\operatorname{\mathbf{R}\mathcal{H}\! \mathit{om}} (\mathbf{R}f_\ast \mathcal{O}_{\widetilde{Y}} (\lfloor c \cdot G \rfloor) , \omega^\bullet_Y)$ occurs in degree $-\dim Y$ and is torsion free of rank one.

    Denote by $\omega^\bullet_{\operatorname{Spec}(\widehat{\mathcal{O}}_{Y,p})} := t^\ast \omega^\bullet_Y$, which is a dualizing complex by \cite[\href{https://stacks.math.columbia.edu/tag/0A7G}{Tag 0A7G} \& \href{https://stacks.math.columbia.edu/tag/0DWD}{Tag 0DWD}]{StacksProject}. There is a string of isomorphisms:
    \begin{displaymath}
        \begin{aligned}
            t^\ast  & \operatorname{\mathbf{R}\mathcal{H}\! \mathit{om}} (\mathbf{R}f_\ast \mathcal{O}_{\widetilde{Y}} (\lfloor c \cdot G \rfloor) , \omega^\bullet_Y) [-\dim Y]
            \\&\cong \operatorname{\mathbf{R}\mathcal{H}\! \mathit{om}} (t^\ast \mathbf{R}f_\ast \mathcal{O}_{\widetilde{Y}} (\lfloor c \cdot G \rfloor) , t^\ast \omega^\bullet_Y)[-\dim Y] && \textrm{(\cite[Proposition 22.70]{Gortz/Wedhorn:2023})}
            \\&\cong \operatorname{\mathbf{R}\mathcal{H}\! \mathit{om}} (t^\ast \mathbf{R}f_\ast \mathcal{O}_{\widetilde{Y}} (\lfloor c \cdot G \rfloor) , \omega^\bullet_{\operatorname{Spec}(\widehat{\mathcal{O}}_{Y,p})})[-\dim Y].
        \end{aligned}
    \end{displaymath}
    Consequently, we know that 
    \begin{displaymath}
        \operatorname{\mathbf{R}\mathcal{H}\! \mathit{om}} (t^\ast \mathbf{R}f_\ast \mathcal{O}_{\widetilde{Y}} (\lfloor c \cdot G \rfloor) , \omega^\bullet_{\operatorname{Spec}(\widehat{\mathcal{O}}_{Y,p})})
    \end{displaymath}
    is concentrated in degree $-\dim Y$. Moreover, as $t$ is a flat morphism between integral schemes \cite[\href{https://stacks.math.columbia.edu/tag/0AXV}{Tag 0AXV}]{StacksProject}, this nontrivial cohomology sheaf is torsion free of rank one.

    Next, from the flatness of $t$, our hypothesis implies $\widehat{\mathcal{O}}_{Y,p} \in \langle t^\ast \mathbf{R} f_\ast \mathcal{O}_{\widetilde{Y}}(\lfloor c \cdot G \rfloor) \rangle_1$ in $D^b_{\operatorname{coh}}(\widehat{\mathcal{O}}_{Y,p})$. As $Y$ is normal, we know that $\widehat{\mathcal{O}}_{Y,p}$ is a normal domain (see e.g.\ \cite[\href{https://stacks.math.columbia.edu/tag/0C23}{Tag 0C23}]{StacksProject}). So, \Cref{lem:krull_schmidt} implies that $\widehat{\mathcal{O}}_{Y,p}$ is a direct summand of $t^\ast \mathbf{R} f_\ast \mathcal{O}_{\widetilde{Y}}(\lfloor c \cdot G \rfloor)$ because $\widehat{\mathcal{O}}_{Y,p}$ is indecomposable. This gives us morphisms $\alpha \colon \widehat{\mathcal{O}}_{Y,p} \to t^\ast \mathbf{R} f_\ast \mathcal{O}_{\widetilde{Y}}(\lfloor c \cdot G \rfloor)$ and $\beta \colon t^\ast \mathbf{R} f_\ast \mathcal{O}_{\widetilde{Y}}(\lfloor c \cdot G \rfloor) \to \widehat{\mathcal{O}}_{Y,p}$ whose composition is the identity on $\widehat{\mathcal{O}}_{Y,p}$. 

    Now, we finish the proof. Set $\alpha^\prime := \operatorname{\mathbf{R}\mathcal{H}\! \mathit{om}}(\alpha,\omega^\bullet_{\operatorname{Spec}(\widehat{\mathcal{O}}_{Y,p})})$ and $\beta^\prime := \operatorname{\mathbf{R}\mathcal{H}\! \mathit{om}}(\beta,\omega^\bullet_{\operatorname{Spec}(\widehat{\mathcal{O}}_{Y,p})})$. Note that $\alpha^\prime \circ \beta^\prime$ is the identity on $\omega^\bullet_{\operatorname{Spec}(\widehat{\mathcal{O}}_{Y,p})}$. Here, $\omega^\bullet_{\operatorname{Spec}(\widehat{\mathcal{O}}_{Y,p})}$ is concentrated in degree $-\dim Y$. 
    Also, from \cite[\href{https://stacks.math.columbia.edu/tag/0AWE}{Tag 0AWE}]{StacksProject}, the stalks of $\mathcal{H}^{-\dim Y}(\omega^\bullet_{\operatorname{Spec}(\widehat{\mathcal{O}}_{Y,p})})$ has depth $\geq 2$ at all prime ideal of $\widehat{\mathcal{O}}_{Y,p}$ with height $\geq 2$.
    Moreover, observe that if we localize at any prime ideal $\mathfrak{q}$ of $\widehat{\mathcal{O}}_{Y,p}$ with height at most one, the corresponding local rings is a DVR or a field. So, for such $\mathfrak{q}$, we know $\mathcal{H}^{-\dim Y}(\beta^\prime)_{\mathfrak{q}}$ must be an isomorphism because it admits a section with source/target being torsion free of rank one. Therefore, by \cite[\href{https://stacks.math.columbia.edu/tag/0AVS}{Tag 0AVS}]{StacksProject}, $\mathcal{H}^{-\dim Y}(\beta^\prime)$ is an isomorphism. Consequently, after dualizing once more, we see that $\beta$ is an isomorphism, which completes the proof.
\end{proof}

We now turn towards the proof of \Cref{prop:pairs} that $\mathcal{O}_Y\in\langle\mathcal{J}(Y,\mathcal{I^c})\rangle_1$ for $Y$ a normal quasi-affine variety with locally complete intersection singularities. We begin with a lemma.

\begin{lemma}
    \label{lem:lci_rationality_compact}
    Consider a log resolution $f\colon \widetilde{Y}\to Y$ of $(Y,\mathcal{I}^c)$ and $Q\in \operatorname{Perf}(Y)$ such that $\operatorname{Perf}(Y) = \langle Q \rangle$ (see e.g.\ \cite[Theorem 3.1.1]{Bondal/VandenBergh:2003}). If $Y$ has locally complete intersection singularities, then $\mathcal{O}_Y \in \langle Q  \otimes^{\mathbf{L}} \mathbf{R} f_\ast \mathcal{O}_{\widetilde{Y}}(\lfloor c \cdot G \rfloor)  \rangle$.
\end{lemma}

\begin{proof}
    Let $p\in Y$. Denote by $t\colon \operatorname{Spec}(\mathcal{O}_{Y,p}) \to Y$ the natural morphism. We know that $\mathcal{O}_{Y,p}$ is the affine spectrum of a local complete intersection. So, by \cite[Theorem 5.2]{Pollitz:2019}, there is a $P\in \operatorname{Perf}(\mathcal{O}_{Y,p})$ such that $\operatorname{Supp}(P)=\operatorname{Supp} (t^\ast \mathbf{R} f_\ast \mathcal{O}_{\widetilde{Y}}(\lfloor c \cdot G \rfloor))$ and $P\in \langle t^\ast \mathbf{R} f_\ast \mathcal{O}_{\widetilde{Y}}(\lfloor c \cdot G \rfloor) \rangle$. Since $t^\ast \mathbf{R} f_\ast \mathcal{O}_{\widetilde{Y}}(\lfloor c \cdot G \rfloor)$ has full support, \cite[Lemma 1.2]{Neeman:1992} tells us $\mathcal{O}_{Y,p} \in \langle t^\ast \mathbf{R} f_\ast \mathcal{O}_{\widetilde{Y}}(\lfloor c \cdot G \rfloor) \rangle$ because $\mathcal{O}_{Y,p} \in \langle P \rangle$. Then, by \cite[Theorem 1.7]{BILMP:2023}, we are done.
\end{proof}

\begin{corollary}
    \label{cor:lci_rationality}
    Suppose $Y$ admits a very ample line bundle $\mathcal{L}$. Consider a log resolution $f\colon \widetilde{Y}\to Y$ of $(Y,\mathcal{I}^c)$. If $Y$ has locally complete intersection singularities, then there is a smallest $N\geq 0$ such that
    \begin{displaymath}
        \mathcal{O}_Y \in \langle \bigoplus^{\dim Y }_{i=0} \mathcal{L}^{\otimes i} \otimes^{\mathbf{L}} \mathbf{R} f_\ast \mathcal{O}_{\widetilde{Y}}(\lfloor c \cdot G \rfloor)  \rangle_{N+1}.
    \end{displaymath}
    Moreover, in such cases, $N=0$ if, and only if, $(Y,\mathcal{I}^c)$ has rational singularities. Additionally, if $Y$ is quasi-affine, then the value $\operatorname{level}^{\mathbf{R} f_\ast \mathcal{O}_{\widetilde{Y}}(\lfloor c \cdot G \rfloor)} (\mathcal{O}_Y)$ is finite. 
\end{corollary}

\begin{proof}
    The first claim is a special case of \Cref{lem:lci_rationality_compact} coupled with \cite[Theorem 4]{Orlov:2009}. Indeed, loc.\ cit.\ ensures that we can take $Q= \oplus^{\dim Y}_{i=0} \mathcal{L}^{\otimes i}$. 
    
    Now, we check the second claim. If $(Y,\mathcal{I}^c)$ has rational singularities, then \Cref{thm:schwede_takagi_pairs} implies $N=0$. Conversely, if $N=0$, then for each $p\in Y$, we see that
    \begin{displaymath}
       \mathcal{O}_{Y,p} \in \langle t_p^\ast (\bigoplus^{\dim Y }_{i=0} \mathcal{L}^{\otimes i} \otimes^{\mathbf{L}} \mathbf{R} f_\ast \mathcal{O}_{\widetilde{Y}}(\lfloor c \cdot G \rfloor) ) \rangle_1
    \end{displaymath}
    where $t_p\colon \operatorname{Spec}(\mathcal{O}_{Y,p}) \to Y$ is the natural morphism. However, as $\mathcal{L}$ is a line bundle, we have an isomorphism
    \begin{displaymath}
        t_p^\ast (\bigoplus^{\dim Y }_{i=0} \mathcal{L}^{\otimes i} \otimes^{\mathbf{L}} \mathbf{R} f_\ast \mathcal{O}_{\widetilde{Y}}(\lfloor c \cdot G \rfloor) ) \cong t_p^\ast \mathbf{R} f_\ast \mathcal{O}_{\widetilde{Y}}(\lfloor c \cdot G \rfloor)^{\oplus 1+ \dim Y}.
    \end{displaymath}
    Hence, $\mathcal{O}_{Y,p} \in \langle t_p^\ast \mathbf{R} f_\ast \mathcal{O}_{\widetilde{Y}}(\lfloor c \cdot G \rfloor) \rangle_1$. By base changing along the natural morphism $\operatorname{Spec}(\widehat{\mathcal{O}}_{Y,p}) \to \operatorname{Spec}(\mathcal{O}_{Y,p})$, we can argue exactly as in \Cref{thm:schwede_takagi_pairs} to show $\mathcal{O}_Y \in \langle \mathbf{R} f_\ast \mathcal{O}_{\widetilde{Y}}(\lfloor c \cdot G \rfloor) \rangle_1$, which tells us that $(Y,\mathcal{I}^c)$ has rational singularities.

    That the last claim holds follows from \Cref{lem:lci_rationality_compact} because $\langle \mathcal{O}_Y\rangle = \operatorname{Perf}(Y)$ (see e.g.\ \cite[\href{https://stacks.math.columbia.edu/tag/0BQQ}{Tag 0BQQ}]{StacksProject}). 
\end{proof}

\begin{lemma}
    \label{lem:rationality_CM_via_multiplier_submodule}
    Suppose $Y$ is Cohen--Macaulay. Consider a log resolution $f\colon \widetilde{Y}\to Y$ of $(Y,\mathcal{I}^c)$. Then the value $\operatorname{level}^{\mathbf{R} f_\ast \mathcal{O}_{\widetilde{Y}}(\lfloor c \cdot G \rfloor)} (\mathcal{O}_Y)$ coincides with $\operatorname{level}^{\mathcal{J}(\omega_Y,\mathcal{I}^c)} (\omega_Y)$.
\end{lemma}

\begin{proof}
    This comes directly from a few standard facts, which we spell out. There is a string of isomorphisms:
    \begin{displaymath}
        \begin{aligned}
            \mathcal{J}(\omega_Y,\mathcal{I}^c)
            &\cong \mathbf{R} f_\ast (\mathcal{O}_{\widetilde{Y}}(- \lfloor c \cdot G \rfloor) \otimes^{\mathbf{L}} \omega_{\widetilde{Y}}) && \textrm{(\cite[Lemma 3.5]{Schwede/Takagi:2008})}
            \\&\cong \operatorname{\mathbf{R}\mathcal{H}\! \mathit{om}} (\mathbf{R}f_\ast \mathcal{O}_{\widetilde{Y}} (\lfloor c \cdot G \rfloor) , \omega^\bullet_Y) [ - \dim Y].
        \end{aligned}
    \end{displaymath}
    Moreover, by duality, we have a triangulated equivalence
    \begin{equation}
        \label{eq:rational_entropy_properties}
        \operatorname{\mathbf{R}\mathcal{H}\! \mathit{om}} ( -, \omega_Y^\bullet) \colon D^b_{\operatorname{coh}}(Y) \to \big( D^b_{\operatorname{coh}}(Y) \big)^{op}
    \end{equation}
    with the opposite category. Also, being that $Y$ is Cohen--Macaulay, we know that $\omega^\bullet_Y = \omega_Y[\dim Y]$. Therefore, using \Cref{eq:rational_entropy_properties}, the desired claim follows from the string of equalities:
    \begin{displaymath}
        \begin{aligned}
            \operatorname{level}^{\mathbf{R} f_\ast \mathcal{O}_{\widetilde{Y}}(\lfloor c \cdot G \rfloor)} (\mathcal{O}_Y) 
            &= \operatorname{level}^{\mathcal{J}(\omega_Y,\mathcal{I}^c)} (\omega_Y) && (\textrm{in } \big( D^b_{\operatorname{coh}}(Y) \big)^{op})
            \\&= \operatorname{level}^{\mathcal{J}(\omega_Y,\mathcal{I}^c)} (\omega_Y) && (\textrm{in } D^b_{\operatorname{coh}}(Y))
        \end{aligned}
    \end{displaymath}
    where the last equality uses \cite[Lemma 2.4(5)]{Avramov/Buchweitz/Iyengar/Miller:2010}.
\end{proof}

\begin{proof}
    [Proof of \Cref{prop:pairs}]
    As $Y$ has local complete intersection singularities, it is Cohen--Macaulay, and hence this is immediate from \Cref{thm:schwede_takagi_pairs}, \Cref{cor:lci_rationality}, and \Cref{lem:rationality_CM_via_multiplier_submodule}.
\end{proof}

\begin{appendix}

\section{Koll\'{a}r--Kov\'{a}cs rational pairs}
\label{app:kollar_kovacs_pairs}

We briefly discuss variations of our results in the context of rational pairs due to Koll\'{a}r--Kov\'{a}cs \cite{Kollar:2013}. 

\subsection{Reminder}
\label{app:kollar_kovacs_pairs_reminder}

Recall that a \textbf{reduced pair} consists of a normal variety $Y$ over a field $k$ of characteristic zero and a Weil divisor $D$ on $Y$ whose coefficients are all one. This datum is denoted by $(Y,D)$. A \textbf{thrifty resolution} for $(Y,D)$ is a proper birational morphism $f\colon \widetilde{Y}\to Y$ from a smooth variety $\widetilde{Y}$ (over $k$) such that the birational transform\footnote{$D_Y$ is defined as the closure in $Y$ of the image of $D$ under the inverse function $f^{-1}$ of $f$ defined over the largest open subset $U$ of $X$ which $f$ is an isomorphism over (see \cite[Definition 1.11]{Kollar:2013} or \cite[Definition 1.1]{Erickson:2014}).} $D_Y$ of $D$ is a simple normal crossing divisor, $\operatorname{exc}(f)$ does not contain any stratum of $(Y,D_Y)$, and $f(\operatorname{exc}(f))$ does not contain any stratum of the simple normal crossing divisor locus of $(Y,D_Y)$. Recall that a \textbf{stratum}, in the sense of \cite[Definition 1.7]{Kollar:2013}, is an irreducible component of $\bigcap_{i\in I} D_i$ where $I$ is a subset of $J$ and $D_Y=\sum_{j\in J} D_j$. Also, the simple normal crossing divisor locus of $(Y,D_Y)$ is the largest open subscheme $U$ of $Y$ such that $(U,D|_U)$ is a simple normal crossing divisor pair in the sense of \cite[Definition 1.7]{Kollar:2013}.

We say a reduced pair $(Y,D)$ has \textbf{rational singularities \`{a} la Koll\'{a}r--Kov\'{a}cs} if the natural map $\mathcal{O}_Y (-D)\xrightarrow{ntrl.} \mathbf{R}f_\ast \mathcal{O}_{\widetilde{Y}} (-D_Y)$ is an isomorphism for some thrifty resolution $f\colon \widetilde{Y} \to Y$ of $(Y,D)$. By \cite[Corollary 2.86]{Kollar:2013}, this notion is independent of the thrifty resolution of the reduced pair. Note that the notion of a Koll\'{a}r--Kov\'{a}cs rational pair does not necessarily coincide with the notion of a Schwede--Takagi rational pair explored in this paper. For example, Koll\'{a}r--Kov\'{a}cs rationality does not neccessarily imply the ambient variety has rational singularities (e.g.\ see \cite[Remark 2.81]{Kollar:2013}). On the other hand, by \cite[Corollary 3.13]{Schwede/Takagi:2008}, Schwede--Takagi rationality enjoys this property.

\subsection{Variations}
\label{app:kollar_kovacs_pairs_variations}

Let $(Y,D)$ be a reduced pair. We state a `splitting criteria' for having rational singularities \`{a} la Koll\'{a}r--Kov\'{a}cs in this context. The following is an expected variation of \cite[Theorem 1]{Kovacs:2000}, \cite[Theorem 2.12]{Bhatt:2012} and \cite[Theorem 9.5]{Murayama:2021} in this setting. Its proof follows in a similar vein, so we only give a sketch. In fact, this is known to experts, but we record it for lack of reference.

\begin{proposition}
    \label{prop:kovacs_splitting_rational_pairs}
    A reduced pair $(Y,D)$ is rational \`{a} la Koll\'{a}r--Kov\'{a}cs if, and only if, the morphism $\mathcal{O}_Y (-D)\xrightarrow{ntrl.} \mathbf{R} f_\ast \mathcal{O}_{\widetilde{Y}} (-D_Y)$ splits for some thrifty resolution $ f\colon \widetilde{Y}\to Y$ of $(Y,D)$.
\end{proposition}

\begin{proof}
    It is clear that $(Y,D)$ being rational \`{a} la Koll\'{a}r--Kov\'{a}cs implies $\mathcal{O}_Y (-D)\xrightarrow{ntrl.} \mathbf{R} f_\ast \mathcal{O}_{\widetilde{Y}} (-D_Y)$ splits for some thrifty resolution $f\colon \widetilde{Y}\to Y$ of $(Y,D)$. We check the converse. Suppose there is a thrifty resolution $f\colon \widetilde{Y}\to Y$ of $(Y,D)$ such that $\xi\colon \mathcal{O}_Y (-D) \xrightarrow{ntrl.} \mathbf{R} f_\ast \mathcal{O}_{\widetilde{Y}} (-D_Y)$ splits in $D^b_{\operatorname{coh}}(Y)$. Recall the following vanishing for $j>0$ holds (see e.g.\ \cite[Lemma 2.5]{Kovacs/Schwede:2016}, \cite[Theorem 10.39]{Kollar:2013}, or \cite[Proposition 3.6]{Erickson:2014}):
    \begin{displaymath}
        \begin{aligned}
            0&= \mathbf{R}^j  f_\ast (\omega_{\widetilde{Y}} \otimes^\mathbf{L} \mathcal{O}_{\widetilde{Y}} (D_Y)) \\&\cong \mathbf{R}^j  f_\ast \big( \operatorname{\mathbf{R}\mathcal{H}\! \mathit{om}}_{\widetilde{Y}}(\mathcal{O}_{\widetilde{Y}} (-D_Y), \omega_{\widetilde{Y}}) \big).
        \end{aligned}
    \end{displaymath}
    Then the desired claim holds by \cite[Theorem 2.74]{Kollar:2013}.
\end{proof}

\begin{theorem}
    \label{thm:kollar_kovac_pairs}
    A reduced pair $(Y,D)$ is rational \`{a} la Koll\'{a}r--Kov\'{a}cs if, and only if, $\mathcal{O}_Y (-D)$ is an object of $\langle \mathbf{R} f_\ast \mathcal{O}_Y (-D_Y)\rangle_1$ for some thrifty resolution $f\colon \widetilde{Y} \to Y$ of $(Y,D)$.
\end{theorem}

\begin{proof}
    This follows from an argument similar to that in the proof of \Cref{thm:schwede_takagi_pairs}, combined with \Cref{prop:kovacs_splitting_rational_pairs}. 
\end{proof}

\end{appendix}

\bibliographystyle{alpha}
\bibliography{mainbib}

@article {Mathew:2016,
    AUTHOR = {Mathew, Akhil},
    TITLE = {The {G}alois group of a stable homotopy theory},
    JOURNAL = {Adv. Math.},
    FJOURNAL = {Advances in Mathematics},
    VOLUME = {291},
    YEAR = {2016},
    PAGES = {403--541},
    ISSN = {0001-8708,1090-2082},
    MRCLASS = {14F35 (11F11 11S20 14F20 18D10 18G55 55P43 55U35)},
    MRNUMBER = {3459022},
    MRREVIEWER = {Rui\ Miguel\ Saramago},
    DOI = {10.1016/j.aim.2015.12.017},
    URL = {https://doi.org/10.1016/j.aim.2015.12.017},
}

@Article{Schwede/Takagi:2008,
    Author = {Schwede, Karl and Takagi, Shunsuke},
    Title = {Rational singularities associated to pairs},
    FJournal = {Michigan Mathematical Journal},
    Journal = {Mich. Math. J.},
    ISSN = {0026-2285},
    Volume = {57},
    Pages = {625--658},
    Year = {2008},
    Language = {English},
    DOI = {10.1307/mmj/1220879429},
    zbMATH = {5604552},
    Zbl = {1177.14028}
}

@misc{StacksProject,
    shorthand    = {Stacks},
    author       = {The {Stacks Project Authors}},
    title        = {\textit{Stacks Project}},
    howpublished = {\url{https://stacks.math.columbia.edu}},
    year         = {2026},
}

@article{Lank/Venkatesh:2026,
    author = {Lank, Pat and Venkatesh, Sridhar},
    title = {Triangulated characterizations of singularities},
    fjournal = {Nagoya Mathematical Journal},
    journal = {Nagoya Math. J.},
    issn = {0027-7630},
    volume = {261},
    pages = {15},
    note = {Id/No e3},
    year = {2026},
    language = {English},
    doi = {10.1017/nmj.2025.11},
    zbMATH = {8135530}
}

@Book{Kollar:2013,
    Author = {Koll{\'a}r, J{\'a}nos},
    Title = {Singularities of the minimal model program. {With} the collaboration of {S{\'a}ndor} {Kov{\'a}cs}},
    FSeries = {Cambridge Tracts in Mathematics},
    Series = {Camb. Tracts Math.},
    ISSN = {0950-6284},
    Volume = {200},
    ISBN = {978-1-107-03534-8; 978-1-139-54789-5},
    Year = {2013},
    Publisher = {Cambridge: Cambridge University Press},
    Language = {English},
    DOI = {10.1017/CBO9781139547895.002},
    zbMATH = {6148846},
    Zbl = {1282.14028}
}

@Book{EGAI:1971,
    Author = {Grothendieck, Alexander and Dieudonn{\'e}, Jean A.},
    Title = {{\'E}l{\'e}ments de g{\'e}om{\'e}trie alg{\'e}brique. {I}},
    FSeries = {Grundlehren der Mathematischen Wissenschaften},
    Series = {Grundlehren Math. Wiss.},
    ISSN = {0072-7830},
    Volume = {166},
    ISBN = {3-540-05113-9},
    Year = {1971},
    Publisher = {Springer, Cham},
    Language = {English},
    zbMATH = {3322751},
    Zbl = {0203.23301}
}

@Book{Gortz/Wedhorn:2023,
    Author = {G{\"o}rtz, Ulrich and Wedhorn, Torsten},
    Title = {Algebraic geometry {II}: cohomology of schemes. {With} examples and exercises},
    FSeries = {Springer Studium Mathematik -- Master},
    Series = {Springer Stud. Math. -- Master},
    ISSN = {2509-9310},
    ISBN = {978-3-658-43030-6; 978-3-658-43031-3},
    Year = {2023},
    Publisher = {Wiesbaden: Springer Spektrum},
    Language = {English},
    DOI = {10.1007/978-3-658-43031-3},
    zbMATH = {7802900}
}

@Article{Kovacs/Schwede:2016,
    Author = {Kov{\'a}cs, S{\'a}ndor and Schwede, Karl},
    Title = {Inversion of adjunction for rational and {Du} {Bois} pairs},
    FJournal = {Algebra \& Number Theory},
    Journal = {Algebra Number Theory},
    ISSN = {1937-0652},
    Volume = {10},
    Number = {5},
    Pages = {969--1000},
    Year = {2016},
    Language = {English},
    DOI = {10.2140/ant.2016.10.969},
    zbMATH = {6617176},
    Zbl = {1349.14127}
}

@Book{Kollar:1997,
    Author = {Koll{\'a}r, J{\'a}nos},
    Title = {Singularities of pairs, Algebraic geometry. {Proceedings} of the {Summer} {Research} {Institute}, {Santa} {Cruz}, {CA}, {USA}, {July} 9--29, 1995},
    FSeries = {Proceedings of Symposia in Pure Mathematics},
    Series = {Proc. Symp. Pure Math.},
    ISSN = {0082-0717},
    Volume = {62, 1},
    ISBN = {0-8218-0894-X},
    Year = {1997},
    Publisher = {Providence, RI: American Mathematical Society},
    Language = {English},
    zbMATH = {1109081},
    Zbl = {0882.00032}
}

@Book{Kollar/Mori:2008,
    Author = {Koll{\'a}r, J{\'a}nos and Mori, Shigefumi},
    Title = {Birational geometry of algebraic varieties. {With} the collaboration of {C}. {H}. {Clemens} and {A}. {Corti}},
    Edition = {Paperback reprint of the hardback edition 1998},
    FSeries = {Cambridge Tracts in Mathematics},
    Series = {Camb. Tracts Math.},
    ISSN = {0950-6284},
    Volume = {134},
    ISBN = {978-0-521-06022-6},
    Year = {2008},
    Publisher = {Cambridge: Cambridge University Press},
    Language = {English},
    zbMATH = {5273473},
    Zbl = {1143.14014}
}

@Book{Hartshorne:1966,
    Author = {Hartshorne, Robin},
    Title = {Residues and duality. {Lecture} notes of a seminar on the work of {A}. {Grothendieck}, given at {Havard} 1963/64. {Appendix}: {Cohomology} with supports and the construction of the {{\(f^!\)}} functor by {P}. {Deligne}},
    FSeries = {Lecture Notes in Mathematics},
    Series = {Lect. Notes Math.},
    ISSN = {0075-8434},
    Volume = {20},
    Year = {1966},
    Publisher = {Springer, Cham},
    Language = {English},
    DOI = {10.1007/BFb0080482},
    URL = {https://eudml.org/doc/203789},
    zbMATH = {3336606},
    Zbl = {0212.26101}
}

@Article{Kovacs:2000,
    Author = {Kov{\'a}cs, S{\'a}ndor J.},
    Title = {A characterization of rational singularities},
    FJournal = {Duke Mathematical Journal},
    Journal = {Duke Math. J.},
    ISSN = {0012-7094},
    Volume = {102},
    Number = {2},
    Pages = {187--191},
    Year = {2000},
    Language = {English},
    DOI = {10.1215/S0012-7094-00-10221-9},
    zbMATH = {1436088},
    Zbl = {0973.14001}
}

@article{Murayama:2021,
    author = {Murayama, Takumi},
    title = {Relative vanishing theorems for {{\(\mathbb{Q}\)}}-schemes},
    fjournal = {Algebraic Geometry},
    journal = {Algebr. Geom.},
    issn = {2313-1691},
    volume = {12},
    number = {1},
    pages = {84--144},
    year = {2025},
    language = {English},
    doi = {10.14231/AG-2025-003},
    zbMATH = {7959541},
    Zbl = {1565.14040}
}

@article{DeDeyn/Lank/Lank/ManaliRahul/Venkatesh:2026,
    author = {{De Deyn}, Timothy and Lank, Pat and {Manali Rahul}, Kabeer and Venkatesh, Sridhar},
    title = {Measuring birational derived splinters},
    journal = {Bulletin of the London Mathematical Society},
    volume = {58},
    number = {5},
    pages = {e70362},
    doi = {https://doi.org/10.1112/blms.70362},
    url = {https://londmathsoc.onlinelibrary.wiley.com/doi/abs/10.1112/blms.70362},
    eprint = {https://londmathsoc.onlinelibrary.wiley.com/doi/pdf/10.1112/blms.70362},
    year = {2026}
}

@article {Bondal/VandenBergh:2003,
    AUTHOR = {Bondal, Alexei and {V}an den {B}ergh, Michel},
    TITLE = {Generators and representability of functors in commutative and
    noncommutative geometry},
    JOURNAL = {Mosc. Math. J.},
    FJOURNAL = {Moscow Mathematical Journal},
    VOLUME = {3},
    YEAR = {2003},
    NUMBER = {1},
    PAGES = {1--36, 258},
    ISSN = {1609-3321,1609-4514},
    MRCLASS = {18E30 (14F05)},
    MRNUMBER = {1996800},
    MRREVIEWER = {Ioannis\ Emmanouil},
    DOI = {10.17323/1609-4514-2003-3-1-1-36},
    URL = {https://doi.org/10.17323/1609-4514-2003-3-1-1-36},
}

@Article{Avramov/Buchweitz/Iyengar/Miller:2010,
    Author = {Avramov, Luchezar L. and Buchweitz, Ragnar-Olaf and Iyengar, Srikanth B. and Miller, Claudia},
    Title = {Homology of perfect complexes},
    FJournal = {Advances in Mathematics},
    Journal = {Adv. Math.},
    ISSN = {0001-8708},
    Volume = {223},
    Number = {5},
    Pages = {1731--1781},
    Year = {2010},
    Language = {English},
    DOI = {10.1016/j.aim.2009.10.009},
    zbMATH = {5678601},
    Zbl = {1186.13006}
}

@article{BILMP:2023,
    author = {Ballard, Matthew R. and Iyengar, Srikanth B. and Lank, Pat and Mukhopadhyay, Alapan and Pollitz, Josh},
    title = {High {Frobenius} pushforwards generate the bounded derived category},
    fjournal = {Forum of Mathematics, Sigma},
    journal = {Forum Math. Sigma},
    issn = {2050-5094},
    volume = {14},
    pages = {29},
    note = {Id/No e12},
    year = {2026},
    language = {English},
    doi = {10.1017/fms.2025.10156},
    zbMATH = {8155611}
}

@article{Pollitz:2019,
    author = {Pollitz, Josh},
    title = {The derived category of a locally complete intersection ring},
    fjournal = {Advances in Mathematics},
    journal = {Adv. Math.},
    issn = {0001-8708},
    volume = {354},
    pages = {18},
    note = {Id/No 106752},
    year = {2019},
    language = {English},
    doi = {10.1016/j.aim.2019.106752},
    zbMATH = {7103871},
    Zbl = {1428.13022}
}

@article{Balmer:2016,
    author = {Balmer, Paul},
    title = {Separable extensions in tensor-triangular geometry and generalized {Quillen} stratification},
    fjournal = {Annales Scientifiques de l'{\'E}cole Normale Sup{\'e}rieure. Quatri{\`e}me S{\'e}rie},
    journal = {Ann. Sci. {\'E}c. Norm. Sup{\'e}r. (4)},
    issn = {0012-9593},
    volume = {49},
    number = {4},
    pages = {907--925},
    year = {2016},
    language = {English},
    doi = {10.24033/asens.2298},
    url = {smf4.emath.fr/en/Publications/AnnalesENS/4_49/html/ens_ann-sc_49_907-925.php},
    zbMATH = {6680008},
    Zbl = {1376.18004}
}

@article {Atiyah:1956,
    AUTHOR = {Atiyah, M.},
    TITLE = {On the {K}rull-{S}chmidt theorem with application to sheaves},
    JOURNAL = {Bull. Soc. Math. France},
    FJOURNAL = {Bulletin de la Soci\'{e}t\'{e} Math\'{e}matique de France},
    VOLUME = {84},
    YEAR = {1956},
    PAGES = {307--317},
    ISSN = {0037-9484},
    MRCLASS = {53.3X},
    MRNUMBER = {86358},
    MRREVIEWER = {S.\ Eilenberg},
    URL = {http://www.numdam.org/item?id=BSMF_1956__84__307_0},
}

@Article{Walker/Warfield:1976,
    Author = {Walker, C. L. and Warfield, R. B. jun.},
    Title = {Unique decomposition and isomorphic refinement theorems in additive categories},
    FJournal = {Journal of Pure and Applied Algebra},
    Journal = {J. Pure Appl. Algebra},
    ISSN = {0022-4049},
    Volume = {7},
    Pages = {347--359},
    Year = {1976},
    Language = {English},
    DOI = {10.1016/0022-4049(76)90059-1},
    zbMATH = {3550925},
    Zbl = {0354.18011}
}

@Article{Krause:2015,
    Author = {Krause, Henning},
    Title = {Krull-{Schmidt} categories and projective covers},
    FJournal = {Expositiones Mathematicae},
    Journal = {Expo. Math.},
    ISSN = {0723-0869},
    Volume = {33},
    Number = {4},
    Pages = {535--549},
    Year = {2015},
    Language = {English},
    DOI = {10.1016/j.exmath.2015.10.001},
    zbMATH = {6523697},
    Zbl = {1353.18011}
}

@Article{Le/Chen:2007,
    Author = {Le, Jue and Chen, Xiao-Wu},
    Title = {Karoubianness of a triangulated category},
    FJournal = {Journal of Algebra},
    Journal = {J. Algebra},
    ISSN = {0021-8693},
    Volume = {310},
    Number = {1},
    Pages = {452--457},
    Year = {2007},
    Language = {English},
    DOI = {10.1016/j.jalgebra.2006.11.027},
    zbMATH = {5144403},
    Zbl = {1112.18009}
}

@Article{Bhatt:2012,
    Author = {Bhatt, Bhargav},
    Title = {Derived splinters in positive characteristic},
    FJournal = {Compositio Mathematica},
    Journal = {Compos. Math.},
    ISSN = {0010-437X},
    Volume = {148},
    Number = {6},
    Pages = {1757--1786},
    Year = {2012},
    Language = {English},
    DOI = {10.1112/S0010437X12000309},
    zbMATH = {6147342},
    Zbl = {1291.14036}
}

@misc{Erickson:2014,
    title={Deformation invariance of rational pairs},
    author={Lindsay Erickson},
    year={2014},
    url={https://arxiv.org/abs/1407.0110},
    eprint={1407.0110},
    archivePrefix={arXiv.AG},
    howpublished    = {\href{https://arxiv.org/abs/1407.0110}{arXiv:1407.0110}},
    publisher     = {arXiv},
}

@article{Neeman:1992,
    author = {Neeman, Amnon},
    title = {The chromatic tower for {{\(D(R)\)}}. {With} an appendix by {Marcel} {B{\"o}kstedt}},
    fjournal = {Topology},
    journal = {Topology},
    issn = {0040-9383},
    volume = {31},
    number = {3},
    pages = {519--532},
    year = {1992},
    language = {English},
    doi = {10.1016/0040-9383(92)90047-L},
    zbMATH = {166187},
    Zbl = {0793.18008}
}

@article{Orlov:2009,
    author = {Orlov, Dmitri},
    title = {Remarks on generators and dimensions of triangulated categories},
    fjournal = {Moscow Mathematical Journal},
    journal = {Mosc. Math. J.},
    issn = {1609-3321},
    volume = {9},
    number = {1},
    pages = {143--149},
    year = {2009},
    language = {English},
    url = {www.ams.org/distribution/mmj/vol9-1-2009/abst9-1-2009.html},
    zbMATH = {5642253},
    Zbl = {1197.18004}
}

@Misc{Grifo/Letz/Pollitz:2025,
    author = {Elo\'{i}sa Grifo and Janina Letz and Josh Pollitz},
    title = {{ThickSubcategories}},
    year = {2025},
    howpublished = {A \emph{Macaulay2} package available at \url{https://github.com/eloisagrifo/levels}}}

\end{document}